\documentclass[12pt, a4paper]{amsart}
\usepackage{amssymb,amscd, hyperref, epsf}

\begin{document}

\let\kappa=\varkappa
\let\epsilon=\varepsilon
\let\phi=\varphi
\let\p\partial

\def\Z{\mathbb Z}
\def\R{\mathbb R}
\def\C{\mathbb C}
\def\Q{\mathbb Q}
\def\P{\mathbb P}
\def\N{\mathbb N}
\def\L{\mathbb L}
\def\HH{\mathrm{H}}
\def\ss{X}

\def\conj{\overline}
\def\Beta{\mathrm{B}}
\def\const{\mathrm{const}}
\def\ov{\overline}
\def\wt{\widetilde}
\def\wh{\widehat}

\renewcommand{\Im}{\mathop{\mathrm{Im}}\nolimits}
\renewcommand{\Re}{\mathop{\mathrm{Re}}\nolimits}
\newcommand{\codim}{\mathop{\mathrm{codim}}\nolimits}
\newcommand{\id}{\mathop{\mathrm{id}}\nolimits}
\newcommand{\Aut}{\mathop{\mathrm{Aut}}\nolimits}
\newcommand{\lk}{\mathop{\mathrm{lk}}\nolimits}
\newcommand{\sign}{\mathop{\mathrm{sign}}\nolimits}
\newcommand{\pt}{\mathop{\mathrm{pt}}\nolimits}
\newcommand{\rk}{\mathop{\mathrm{rk}}\nolimits}
\newcommand{\SKY}{\mathop{\mathrm{SKY}}\nolimits}
\newcommand{\st}{\mathop{\mathrm{st}}\nolimits}
\def\Jet{{\mathcal J}}
\def\FC{{\mathrm{FCrit}}}
\def\sS{{\mathcal S}}
\def\lcan{\lambda_{\mathrm{can}}}
\def\ocan{\omega_{\mathrm{can}}}

\renewcommand{\mod}{\mathrel{\mathrm{mod}}}
\def\ds{\displaystyle}

\newtheorem{mainthm}{Theorem}
\newtheorem{thm}{Theorem}
\newtheorem{construction}[thm]{Construction}
\newtheorem{lem}[thm]{Lemma}
\newtheorem{prop}[thm]{Proposition}
\newtheorem{cor}[thm]{Corollary}

\theoremstyle{definition}
\newtheorem{exm}[thm]{Example}
\newtheorem{rem}{Remark}
\newtheorem{defin}[thm]{Definition}
\newtheorem{remark}[thm]{Remark}
\renewcommand{\thesubsection}{\arabic{subsection}}
\numberwithin{equation}{subsection}

\title[Linear programming problems algorithmically unsolvable] 
{Conditions when the problems of linear programming are algorithmically unsolvable} 
\author[V.~Chernov and V.Chernov]{Viktor Chernov and Vladimir Chernov}
\address{V.~Chernov, St Petersburg State University of Economics, Department of applied mathematics and economico-mathematical methods \\ Griboedov canal emb., 30--32, St Petersburg 191023, Russia} 
\email{chernov.v@unecon.ru} 
\address{V.~Chernov, 6188 Kemeny Hall, Mathematics Department, Dartmouth College,\\ Hanover NH 03755, USA} 
\email{vladimir.chernov@dartmouth.edu}

\subjclass{Primary 03D78; Secondary 03F60, 90C05}

\begin{abstract} We study the properties of the constructive linear programming problems. The parameters of linear functions in such problems are constructive real numbers. To solve such a problem is to find the optimal plan with the constructive real number components. We show that it is impossible to have an algorithm that solves an arbitrary constructive real programming problem.
\end{abstract}


\maketitle

\section{Introduction}
Linear programming (LP) is a relatively simple and well studied area of Mathematics. It is equipped with the problem solving algorithms and it has multiple theoretical and real life applications.  The models and methods developed in LP are used not only for the analysis and solutions of the LP problems but also for the  development of optimization algorithms and solving problems from the other areas. 

The problems of LP can be considered as a sort of a testing ground, i.e. a set of test problems for construction and analysis of various optimization algorithms.   The analysis of special cases of the LP algorithms is of an interest for problems from other optimization areas.

The algorithms solving LP programs work with real numbers that are the real valued parameters of linear functions. When we talk about the applicability of an algorithm to real numbers we should assume that these numbers are computable so that they themselves are given by some algorithms. The algorithms deal with constructive objects and a sequence of realizations of this idea yields a concept of a constructive real number.

We will show that the algorithmic side of solving even quite simple problems of Linear Programming is far from being simple.

\section{Particularities of Constructive Mathematics.}
The goal of this work is the study of constructive linear programming problems. For the linear programming problems with rational parameters there are well known algorithms of analyzing a problem and solving it. We show that the algorithmic situation is drastically different if some of the parameters of the problem are constructive real numbers (CRNs).

{\it Constructive Mathematics\/} deals with constructive objects that are essentially words in some alphabet.  For example natural numbers are words in the alphabet whose letters are digits. Rational numbers are words in the alphabet that in addition to digits has the division sign and the negative sign with the natural equivalence relation of equality given by bringing rational numbers to the common denominator.

Algorithms are, for example, Turing machines and they transform words into words. The algorithms can be encoded by natural numbers so that from this code and the input number one can get the result of applying this algorithm to the given input. 

The problem of applicability of an algorithm to a given input (the Halting Problem) is generally unsolvable, i.e. there is no algorithm that given the code of the algorithm and input data always tells 1 or 0 depending on whether this algorithm will eventually terminate working on the given input or not.

Constructive Mathematics assumes that the studied objects are constructive meaning that they can be given by words in a special alphabet. The proof of the theorem stating that a constructive object satisfying given criteria does exist assumes the possibility to construct this object. The proof of a statement that looks like a logical disjunction assumes that it is possible to say which of the disjunction terms is true. In particular the  statements giving the law of excluded middle i.e  disjunctions of the form ``$P$ or not $P$'' for a constructive mathematician are not necessarily true statements. To make them true we need to specify the true part of the disjunction.

For example, given a subset $M$ of natural numbers the statement $n\in M$ or $n\not\in M$ in the classical logic is true for every $n.$ However for a constructive mathematician this statement is true only when you can give an algorithm that terminates on each $n\in \N$ and outputs $1$ or $0$ depending on whether $n\in M$ or not. If such an algorithm exists then the set $M$ (and its complement) are called {\it decidable.} Many subsets of $\N$ are not decidable and in general such algorithms detecting whether $n\in M$ or not do not exist.

There are many unsolvable problems related to decidable sets. There is no general algorithm that given a decidable set (defined by its decision algorithm) tells if the set is finite or not. Moreover there is no algorithm that can always tell if a decidable set is empty or not.

Constructive mathematics uses a special constructive logic. If the existence of a constructive object is proved using this logic then from the proof one can deduce the way of constructing the object. 

If we use proof by contradiction to prove the existence of an object then from the view point of constructive logic we did not succeed to prove the existence. Instead we reached a contradiction to the statement that the object does not exist. However the statement that the object can not not exist does not imply that  the object does exist.

However there is an important class of cases when the double negation of the object existence implies that the object indeed does exist. This is when we proved that a certain decidable set of natural numbers (or words in an alphabet) is nonempty. 

Let $M$ be a decidable subset of the set of positive integers so that there exists a decision algorithm $L$ that given any $n\in N$ produces $1$ if $n\in M$ and it produces $0$ if $n\not \in M.$ 

Assume that an element of the set $M$ can not not exist then an element of $M$ does exist i.e. it can be constructed. The algorithm consists of going through all the elements of $\N$ sequentially and using $L$ to verify if the element is in $M$ or not. The fact that an element of $M$ can not not exist means that eventually we will find it using this algorithm. In other words we can say that every nonempty decidable subset of $\N$ is {\it inhabited.} This rule was formulated by A.~A.~Markov~\cite{Markov2} and is called {\it Markov principle.}

Constructive mathematics was developed in the works of A.~A.~Markov \cite{Markov1}, \cite{Markov2}, N.~A.~Shanin~\cite{Shanin} and their followers~\cite{Tseitin1, Tseitin2, Chernov1, Chernov2, Kushner,  Orevkov1, Orevkov2}.
A different but largely similar approach was formed by E.~Bishop~\cite{Bishop} and his followers~\cite{Bauer, Taylor}. Note that Bishop's approach does not allow the use of Markov principle.

A.~Turing in his works~\cite{Turing1, Turing2} formulated the notion close to the modern concept of a Constructive Real Number.

A {\em Constructive Real Number\/} (CRN) is a pair of algorithms $(F,R)$. An algorithm $F$ ({\it called the fundamental sequence\/}) transforms natural numbers into  rational numbers that are the members of a Cauchy sequence. Algorithm $R$ ({\it regulator of convergence}) transforms positive rational numbers into natural numbers and guarantees the convergence in itself of the sequence $F$ so that for every positive rational 
$\epsilon>0$ and every $m,n>R(\epsilon)$ we have $|F(m)-F(n)|<\epsilon.$ 

Algorithm $F$ is step by step creating a sequence of rational approximations to a limit real number while algorithm $R$ provides the sequence of error estimates on these approximations. Note that the limit real value is not a part of this definition and instead we use the pair of algorithms $(F,R).$ 

If a CRN $(F,R)$ is denoted by $x$, then $x(n)$ denotes $F(n).$

The notion of a CRN is a natural clarification of the notion of a computable number. The CRNs possess the natural properties of real numbers in their computable form.

The set of all CRNs is closed under arithmetic operations and under taking $\max$ and $\min.$ The proof of this is similar to the proof of continuity of these operations for the ordinary real numbers.

The relations of equality, strict and nonstrict inequalities on CRNs are introduced in a natural way. Let $x,y$ be two CRNs
\begin{itemize}
\item $x=y$ if for every positive rational $\epsilon$ we can algorithmically produce a natural number $M$ such that for all $m,n>M$ we have $|x(m)-x(n)|<\epsilon;$
\item $x<y$ ($y>x$) if we can produce a positive rational $\epsilon$ and natural $M$ such that for all $m,n>M$ we have $y(m)>x(n)+\epsilon$;
\item $x\leq y$ ($y\geq x$) if for every positive rational $\epsilon$ we can  algorithmically produce a natural $M$ such that for all $m,n>M$ we have $y(m)>x(n)-\epsilon$.
\end{itemize}

These relations have the usual properties. For example, the equality is reflexive, symmetric and transitive. The strict inequality is anti-reflexive, asymmetric and transitive. The nonstrict inequality is reflexive, anti-symmetric and transitive.

However some properties of these relationships are surprising. In particular all the three of these relationships are not decidable in the set of all CRNs. This means that there is no algorithm that finishes working on every pair of CRNs $(x,y)$ and prints $1$ or $0$ depending on whether these numbers are equal or not. Such algorithms do not esist for the other two above relationships as well.

However the relation of strict inequality on the pairs of CRNs is enumerable meaning that one can produce an algorithm that transforms natural numbers to all such pairs or alternatively there is an algorithm $L$ that finishes working on the pair of CRNs $(x,y)$ exactly if $x<y.$

The equality and nonstrict inequality relations are not enumerable, however their negations are enumerable.

Every pair $(x,y)$ of CRNs can not not satisfy one of the three relations $x<y$, $x=y$ or $y<x$ however there is no algorithm that can always tell which one of these three relationships holds for the pair. 

The relation $x<y$ implies $x\leq y$ and $x=y$ implies $x\leq y$ as well. From the relaion $x\leq y$ follows that there can not not be a term of the disjunction $x<y$ or $x=y$ which is true. However it is impossible to tell which term is true, i.e. there is no algorithm that finishes working on every pair $(x,y)$ of CRNs satisfying $x\leq y$ and produces $1$ or $0$ depending on whether $x=y$ or $x<y$ does hold.

If the statement that $x\leq y$ does not hold then $y<x$. If the statements $x=y$ does not hold then $x<y$ or $y<x$ and one can create an algorithm that given any pair of nonequal CRNs decides which one is larger~\cite{Markov2}.

Each one of the relations $<, >, =$ is equivalent to its double negation. In particular if two CRNs can not be not equal then they are equal.

If in a given subset of the set of all CRNs one of this relations is decidable then the other two relations are decidable as well. 

Every CRN is given by a pair of algorithms so from the classical view point the set of all CRNs is countable. However this statement is constructively false and given any sequence of CRNs one can construct a CRNs that is not equal to any members of this sequence.

The set of all CRNs is constructively complete: given any algorithmic sequence of CRNs equipped with an algorithm guaranteeing its convergence in itself, one can construct a CRN that is the limit of this sequence.

The constructive versions of some of the classical mathematical analysis statements are false. For example the constructive version of the statement that a bounded montotonic sequence always has a limit is false. The usual proof of this real analysis fact involves dividing the interval into two equal halves and choosing the half that contains infinitely many elements of the sequence. It is assumed that such a choice can be done and the corresponding disjunction is true.

However it is not clear how to algorithmically determine the needed half of the segment. Moreover the problem is not in the particular proof approach that we considered. There is a constructive analysis counterexample to such a statement in constructive mathematics given by the famous Specker sequence~\cite{Specker}.

In~\cite{Chernov1, Chernov2} the Specker sequence is used to clarify the surprising topological properties of the space of all CRNs.

Different approaches to the introduction of constructive real numbers,  their properties, and relations arising under different approaches to introducing CRNs can be found in~\cite{Kushner, Markov2, Shanin}.

A rational number can be presented as a word in the alphabet of rational numbers construction (consisting of digits, negative  and ratio signs). Equality, strict and  nonstrict comparison relations on the set of rational numbers are decidable. 

Every rational number can be presented in a standard way as a constant sequence equipped with the trivial convergence regulator that prints $1$ for every positive rational number. So all rational numbers belong to the class of CRNs. We shall call such CRN presentations of rational numbers to be the {\it standard real presentation of a rational number.\/} A {\it quasi rational} number is a standard real presentation of a rational number and a CRN equal to it.

Under this correspondence equal rational numbers are mapped to equal quasi-rational numbers. The notions of strict and nonstrict inequalities are also preserved by this correspondence. 

We will distinguish rational and quasi-rational numbers. In particular there is no algorithm that determines if a given CRN is quasi-rational.

A {\it Constructive function\/} is an algorithm that transforms constructive numbers into constructive numbers with the assumption that equal CRNs should be transformed to equal CRNs. Important and sometimes surprising properties of constructive functions can be found in~\cite{Bishop, Tseitin1, Tseitin2, Chernov1, Kushner, Markov1, Markov2, Orevkov1, Orevkov2, Shanin}. In particular all the linear functions of a constructive real number variable with constructive coefficients are constructive functions. Same is true for polynomial functions of a constructive real variable.

Every constructive function is {\em effectively continuous} i.e. one can construct the algorithmic continuity regulator at every constructive point in its domain~\cite{Markov2, Tseitin2}.

At the same time one can construct a constructive function defined on all the CRNs in $[0,1]$ that is not constructively uniformly continuous. In a similar spirit, one can define constructive functions on $[0,1]$ that are not bounded i.e. not attaining their maximal or minimal value at any constructive point.  Examples of this sort are obtained as limits of sequences of locally piece-wise linear functions. 

In this paper we will deal with linear functions of constructive real variable with CRN coefficients.

In particular all the linear functions of a constructive real number variable with constructive coefficients are constructive functions

We will concentrate our attention on the Constructive Linear Programming Problems (CLPP). The parameters of linear functions in the conditions of a CLPP are CRNs. To solve a CLPP means to find an optimal plan with CRN components. 

The CLPP with rational parameters have the usual properties of the linear programming problems~\cite{Gale}. However in the general case when these parameters are CRNs these properties do not hold. 

In particular it is impossible to have an algorithm that computes the optimal plan in a nonempty and bounded domain. Moreover it is impossible to detect the solvability of such a problem even if we have the complete algorithms for computing the parameters. It is impossible to have an algorithm that given a plan verifies if the plan is allowable, and there is no algorithm that given an allowable plan determines if this plan is optimal. 

{\em  In other words the solvability of CLPP is algorithmically unsolvable. \/}

We also study the algorithmic unsolvability of many other properties of CLPP.

The proof of the Theorems in this paper are based on constructing counterexamples -- on constructing families of CLPP for which the desired algorithm does not exist. It could have been that the  reason why this algorithm does not exist is related to the complexity of the examples, however this is not so and the CLPP families in the proofs are rather simple looking.

The proofs show that the reasons why these algorithms do not exist are that the equality and inequality for CRNs are not decidable relations.

\section{Main Results} Let us introduce the constructions that are the basis of the proofs of the following theorems.

Let $S$ be a (partially defined) algorithm that transforms natural numbers into 0 and 1 and that is not extendable to a totally defined algorithm~\cite{Kleene}. Let us note that the Halting problem for this algorithm is undecidable. Also the set of numbers on which $S$ gives output 1 (or output 0) is undecidable.

\begin{construction}\label{construction1} For each natural $n$ we define the following sequence of rational numbers $s_n.$ 
\begin{itemize}
\item
For each $k$ we put $s_n(k)=0$ if the algorithm $S$ did not yet finish working on input $n$ by step $k$; 
\item we put $s_n(k)=-2^{-m}$ if $S$ has finished working on $n$ by step $k$, produced $0$ and $m$ is the step number when $S$ finished working;
 \item we put $s_n(k)=2^{-m}$ if $S$ has finished working on $n$ by step $k$, produced $1$ and $m$ is the step number when $S$ finished working.
\end{itemize}

For each $n$ the sequence $s_n$ converges in itself at a geometric progression speed thus giving a CRN. If $S$ does not ever stop working on $n$ then $s_n=0$ and $s_n$ is respectively bigger and smaller than 0 if $S$ applies to $n$ and produces $1$ and $0$ respectively.
\end{construction}

The proofs of the Theorems below are based on different CLPPs. We tried to make the CLPPs to be very simple so that the reasons for their unsolvability would be crystal clear.

\begin{construction}\label{construction2}
For a given $n$ let $s_n$ be the CRN defined in Construction~\ref{construction1}. We define the CLPP problem $P_n$ to be: $\max x$ under the conditions that $s_n \cdot x=0$ and $0\leq x\leq1.$

The set of allowable plans of this CLPP is bounded and nonempty since $x=0$ is an allowable plan.
\end{construction}

\begin{thm}\label{theorem3} There is no algorithm such that for an arbitrary CLPP with the nonempty bounded region of allowable plans it does 
\begin{enumerate} 
\item compute the optimal plan;
\item compute the extremal value of the cost function.
\end{enumerate}
\end{thm}

\begin{proof} Consider a sequence of CLPP $\{P(n)\}_{n=1}^{\infty}$ given in Construction~\ref{construction2}.  The set of allowable plans of $P(n)$ is nonempty and bounded for every $n.$ 

If algorithm $S$ finishes working on input $n$ then $s_n\neq 0$. In this case the only allowable plan is $x=0$ which hence is the optimal plan. 

If $S$ does not ever finish working on input $n$ then $s_n=0$ and the allowable plans are all the CRN points of the interval $[0,1]$. The optimal plan is $x=1.$

In both cases the optimal value and the optimal plan are equal CRNs.

If there would exist an algorithm that would find an optimal value or an optimal plan for every CLPP then this algorithm would determine if $S$ will finish working on input $n$ but this is not possible.
\end{proof}

The following two Theorems show that there does not exist an algorithm that can always tell given a plan if this plan is optimal or even allowable. The presence of  extra information does not help.

\begin{thm}\label{theorem4}
There is no algorithm that given a CLPP with a nonempty bounded set of allowable plans can always decide if a given plan is allowable.
\end{thm}

\begin{proof} Fix a natural $n$ and consider a CLPP $P(n)$ given by Construction~\ref{construction2}. The set of allowable plans is bounded and nonempty for all $n.$ The plan $x=1$ is allowable exactly when $S$ does not ever finish working on input $n.$ If  an algorithm from Theorem formulation would exist then it would solve the Halting problem for $n$ and this  is impossible. 
\end{proof}

\begin{thm}\label{theorem5}
There is no algorithm that given a CLPP with a nonempty bounded set of allowable plans always decides if a given allowable plan is optimal or not.
\end{thm}

\begin{proof}
Fix $n\in \N$ and consider the CLPP $P(n)$ given by Construction~\ref{construction2}. The set of allowable plans of $P(n)$ is bounded and nonempty for each $n.$ The allowable plan $x=0$ is optimal exactly if $S$ finishes working on input $n$. If the algorithm from the statement of the Theorem would exist then it would solve the Halting problem for the algorithm $S$ which is impossible.
\end{proof}

\begin{construction}\label{construction6}
Let $E$ be a closed interval in the real line. Let $E_0$ and $E_1$ be the left and the right sub-intervals of $E$ each one of length equal to $2/3$ length $E.$ We now define an algorithm $F$ that transforms each CRN in $E$ to $0$ or $1.$ To find $F(x)$ we compute $x$ with the sufficiently high precision (precision better than $1/6$ of the length of $E$ is fine). If this approximation is in $E_0$ then we put $F(x)=0$ and we put $F(x)=1$ in all other cases. 

Note that the algorithm $F$ is not a well defined algorithm on CRNs since it can transform equal CRNs to different answers. 

It will be important that $F$ is defined for all CRNs in $E$ and if $F(x)=0$ then $x\in E_0$, while if $F(x)=1$ then $x\in E_1.$
\end{construction}

Theorem~\ref{theorem3} proves that there are no algorithms that can compute the optimal plan or the optimal value. Note however that if you know the optimal plan then you can compute the optimal value. The next Theorem~\ref{theorem7} shows that the converse statement is false.

\begin{thm}\label{theorem7} There exists an enumerable class of CLPPs for which you can compute the extremal value of the cost function but you can not compute the optimal plan.
\end{thm}

\begin{proof}
Take $H(n)$ with the cost function $\max \bigl((1+s_n)x+(1-s_n)y\bigr)$ under the conditions $x+y\leq 1$, $x\geq 0$, $y\geq 0.$

If the algorithm $S$ finishes working on input $n$ and produces $0$ then $(1+s_n)<(1-s_n)$ and the optimal plan for the CLPP $H(n)$ is the point with coordinates $(0,1).$ 

If the algorithm $F$ finishes working on input $n$ and produces $1$ then $(1+s_n)>(1-s_n)$ and the optimal plan is $(1,0).$

If $S$ never finishes working on input $n$ then $(1+s_n)=(1-s_n)$ and an optimal point is any point of the interval with end points $(1,0)$ and $(0,1).$

Let us apply the algorithm $F$ from Construction~\ref{construction6} to the points of the interval $E.$ 

Assume that there is an algorithm $G$ that transforms each $n\in \N$ to an optimal plan of the CLPP $H(n).$ Then the composition of $F$ and $G$ would be an extension of $S$ to all $n\in \N$ yielding a contradiction. 

Thus the algorithm $G$ does not exist. However the extrenal value $z$ of the cost function does exist. For each $n$ it equals to the maximum of the two CRNs  $\max\{1+s_n, 1-s_n\}$ and hence is a CRN itself. 
\end{proof} 

\begin{thm}\label{theorem8} 
There is no algorithm that can given a CLPP with a nonempty set of allowable plans always decide 
\begin{enumerate}
\item if the set of allowable plans is bounded;
\item if the cost function is bounded;
\item if there does exist an optimal plan.
\end{enumerate}
\end{thm}

\begin{proof}
For a given $n\in \N$ we define a CLPP $Q(n)$ as follows: $\max x$ under the condition that $s_n\cdot x=0.$

The set of allowable plans of this CLPP is nonempty and $x=0$ is an allowable plan. 

The set of allowable plans of $Q(n)$ is bounded if and only if cost function is bounded if and only if there is an optimal plan if and only if $s_n\neq 0.$ This in turn happens exactly when $S$ finishes working on input $n.$ Thus if an algorithm as in the Theorem statement would exist it would solve the Halting problem for $S.$
\end{proof}

\begin{thm}\label{theorem9}
There is no algorithm that given a CLPP with a bounded set of allowable plans decides if the set of allowable plans is empty or not.
\end{thm}

\begin{proof}
For a given $n\in \N$ we define CLPP $R(n)$ as follows: $\max x$ under the condition that $s_n\cdot x=0$ and $x=1.$

The set of allowable plans of such CLPP is nonempty and the optimal plan exists exactly if $s_n=0$ i.e. when $S$ never finishes its work on $n.$ So if the algorithm from the Theorem statement would exist then it would solve the Halting problem for S.
\end{proof}

\begin{thm}\label{theorem10}
There is no algorithm that given any clearly unsolvable CLPP determines the reason for the unsolvability i.e.~if the set of allowable plans is empty or if the cost function is unbounded.
\end{thm}

\begin{proof} For a given $n\in \N$ we define CLPP $T(n)$ as follows: $\max y$ subject to the conditions $s_n\cdot x=0$ and $x=1.$

This is the CLPP from the previous problem with the changed cost function and now we need to find the maximum of $y$ which does not participate in the restrictions on allowable plans. 
The set of allowable plans is nonempty exactly when the cost function is not bounded and this happens exactly when $s_n\neq 0$ i.e. when $S$ finishes working on $n$. So if we can determine the reason for unsolvability of the CLPP $T(n)$ then we can solve the Halting problem for $S.$
\end{proof}

The next Theorem shows that for sets given by a system of linear equations and inequalities with CRN coefficients there is no algorithm that given such a nonempty set constructs one point in the set. In other words the fact that such a set is nonempty does not imply that it can be inhabited.

\begin{thm}\label{theorem11}
There is no algorithm that given any CLPP with a nonempty allowable plan set determines an allowable plan.
\end{thm}

\begin{proof}
We again use the partially defined algorithm $S$ that transforms integers to $0$ or $1$ and it is not extendible to the whole $\N.$ For each $n\in \N$ we define two sequences of rational numbers $a_n$ and $b_n.$ For each $k$ we put 
\begin{itemize}
\item $a_n(k)=b_n(k)=0$ if $S$ did not yet finish working on $n$ by step $k$;
\item $a_n(k)=2^{-m}$ and $b_n(k)=0$ if $S$ finished working on input $n$ by step $k$, produced $0$ and $m$ is the step number when this happened;
\item $a_n(k)=0$ and $b_n(k)=2^{-m}$ if $S$ finished working on input $n$ by step $k$, produced $1$ and $m$ is the step number when this happened.
\end{itemize}

For every $n$ the sequences $a_n(k)$ and $b_n(k)$ are converging in themselves at a speed of a geometric progression so they define two CRNs.

Let $D(n)$ be the set of allowable plans of a CLPP satisfying $(a_n+b_n)\cdot x =a_n$ and $0\leq x\leq 1.$

If $S$ does not terminate on input $n$ then $a_n=b_n=0$ and $D(n)$ contains every $CRN$ between $0$ and $1.$

If $S$ finishes working on $n$ and produces $0$ then $a_n>0$, $b_n=0$ and the only CRN contained in $D(n)$ is $x=1.$

If $S$ finishes working on $n$ and produces $1$ then $a_n=0$, $b_n>0$ and the only CRN contained in $D(n)$ is $x=0.$

Thus for every $n$ the set $D(n)$ is nonempty and there can not not exist an element of it. Assume that for every $n$ the set $D(n)$ can be inhibited that is there is an algorithm $G$ that given any $n\in \N$ constructs a point of $D(n).$ 

Consider an algorithm $F$ from Construction~\ref{construction6} and apply it to the interval $E=[0,1]$. The composition  of $G$ and $F$ would be an everywhere defined algorithm that extends $S$ to the whole of $\N$ but such an extension does not exist by our assumptions.

Thus the algorithm $G$ does not exist.
\end{proof} 

\begin{remark}
Theorem~\ref{theorem11} shows the situation when on one side there can not be an algorithm that for each $n\in \N$ constructs an element of the set $D(n)$ but at the same time the counterexample can not exist since the set $D(n)$ is nonempty.
\end{remark}

\section{Conclusions}
We have presented the proof of algorithmic undecidability of a number of basic properties of CLPP. The basis for this undecidability is the undecidability of the relation of equality and of the comparison of the constructive real valued parameters of such problems.

The Simplex Method (including the method of artificial basis) uses the operations of addition, subtraction, multiplication, division and comparison of the results of the operations. Its usage assume the algorithmic decidability of the comparison and equality. If the parameters of the CLPP are rational numbers or belong to the extension of the field of rational numbers with the decidable properties of equality and total order, then the method allows one to do the necessary computations and conclusions. We proved above that this is not possible for the general CLPP.

Note that for the relations of equality, strict and nonstrict inequality to be decidable it is sufficient that any one of them is decidable.

\section*{Acknowledgement}
The first author is grateful to N.~A.~Shanin for introducing him to the beautiful subject of Constructive Mathematics.

This work was partially supported by a grant from the Simons Foundation
($\# 513272$ to Vladimir Chernov).


\begin{thebibliography}{200}
\bibitem{Bauer} A.~Bauer, P.~Taylor: {\em The Dedekind Reals in Abstract Stone Duality.\/} Mathematical Structures in Computer Science, {\bf 19} (2009)
\bibitem{Bishop} E.~Bishop, D.~Bridges: {\em Constructive analysis} Grundlehren der Mathematis-
chen Wissenschaften [Fundamental Principles of Mathematical Sciences], {\bf 279}
Springer-Verlag, Berlin (1985)
\bibitem{Tseitin1} G.~S.~Ceitin (Tseitin): {\em Algorithmic operators in constructive complete separable
metric spaces} (in Russian) Dokl. Akad. Nauk SSSR {\bf 128} (1959) 49--52. English
translation in Amer. Math. Soc. Transl. 2, {\bf 64} (1967)
\bibitem{Tseitin2}  G.~S.~Ceitin (Tseitin): {\em Mean-value theorems in constructive analysis} (in Rus-
sian) Trudy Mat. Inst. Steklov. {\bf 67} (1962) 362--384. English translation in
Amer. Math. Soc. Transl. 2, {\bf 98} (1971)
\bibitem{Chernov1} V.~P.~Chernov: {\em Locally Constant Constructive Functions and Connectedness of Intervals} Journal of Logic and Computations, Volume 30, Issue 7, October 2020, p. 1425–1428
\bibitem{Chernov2} V.~P.~Chernov: {\em Types of connectedness of the constructive real number} intervals // arxiv.org 25.08.2021 http://arxiv.org/abs/2108.11189 
\bibitem{Gale} D.~Gale: {\em The theory of linear economic models.} McGraw-Hill, New York (1960)
\bibitem{Kleene} S.~C.~Kleene: {\em Introduction to metamathematics.} NY-Toronto (1952)
\bibitem{Kushner} B.~A.~Kushner: {\em Lectures on constructive mathematical analysis} (in Russian)
Monographs in Mathematical Logic and Foundations of Mathematics. Iz-
dat. "Nauka", Moscow, 1973. 447 pp., English translation in Translations of
Mathematical Monographs, 60 American Mathematical Society, Providence,
R.I. (1984). v+346 pp. ISBN: 0-8218-4513-6
\bibitem{Markov1} A.~A.~Markov: {\em On constructive functions} (in Russian) Trudy Mat. Inst. Steklov
{\bf 52} (1958), 315--348. English translation in Amer. Math. Soc. Transl. 2, {\bf 29}
(1963)


\bibitem{Markov2} A.~A.~Markov: {\em On constructive mathematics} (in Russian) Trudy Mat. Inst.
Steklov {\bf 67 }(1962), 8--14. English translation in Amer. Math. Soc. Transl. 2,
{\bf 98} (1971)
\bibitem{Orevkov1} V.~P.~Orevkov: {\em A constructive map of the square into itself, which moves every
constructive point} (in Russian) Dokl. Akad. Nauk SSSR {\bf 152} (1963) 55--58.
English translation in Soviet Math Dokl. 4 (1963)
\bibitem{Orevkov2}
V.~P.~Orevkov: Certain questions of the theory of polynomials with constructive real coefficients (in Russian) Trudy Mat. Inst. Steklov 72 (1964) 462–487. English translation in Amer. Math. Soc. Transl. 2, 100 (1972)

\bibitem{Shanin} N.~A.~Sanin (Shanin): {\em Constructive real numbers and constructive functional
spaces} (in Russian) Trudy Mat. Inst. Steklov {\bf 67} (1962) 15--294. English trans-
lation in Amer. Math. Soc., Providence R.I. (1968)

\bibitem{Specker}
E.~Specker: {\em Nicht konstruktiv beweisbare Satze der Analysis\/} J. Symbolic Logic {\bf  14} no. 3 (1949), 145--158

\bibitem{Taylor} P.~Taylor: {\em A lambda calculus for real analysis.} Journal of Logic \& Analysis 2:5 (2010) 1--115

\bibitem{Turing1}
A.~M.~Turing: {\em On computable numbers, with an application to the Entscheidungsproblem,\/} Proc. Lond. Math. Soc., ser. 2, {\bf 42} (1936), 230--265
[17] A.M. Turing: Corrections, Proc. Lond. Math. Soc., ser. 2, 43 (1937), 544-546


\bibitem{Turing2}
A.~M.~Turing: {\em Corrections,\/}  Proc. Lond. Math. Soc., ser. 2, {\bf 43} (1937), 544--546

\end{thebibliography}
\end{document}